\documentclass[reqno]{amsart}
\usepackage{amsmath,amssymb,etoolbox,mathrsfs,booktabs,float,enumerate,calc,graphicx}

\theoremstyle{plain}
\newtheorem{theorem}{Theorem}
\newtheorem{proposition}[theorem]{Proposition}
\newtheorem{lemma}[theorem]{Lemma}
\newtheorem{corollary}[theorem]{Corollary}

\theoremstyle{definition}
\newtheorem{definition}[theorem]{Definition}
\newtheorem{remark}[theorem]{Remark}
\newtheorem{remark*}{Remark}

\newcommand{\Irr}{\mathrm{Irr}}

\renewcommand{\H}{\mathfrak H}
\newcommand{\h}{\mathfrak h}
\renewcommand{\a}{\mathfrak a}
\renewcommand{\b}{\mathfrak b}
\renewcommand{\c}{\mathfrak c}
\newcommand{\hx}{\mathfrak h}
\newcommand{\ax}{\mathfrak a}
\newcommand{\bx}{\mathfrak b}
\newcommand{\cx}{\mathfrak c}
\newcommand{\dx}{\mathfrak d}
\newcommand{\kx}{\mathfrak k}
\newcommand{\sq}{\mathfrak k}
\numberwithin{equation}{section}
\numberwithin{theorem}{section}
\begin{document}
\title[Character and class parameters from character values]{Character and class parameters from entries of character tables of symmetric groups}
\address{Center for Communications Research, Princeton}
\author[A.\ R.\ Miller]{Alexander Rossi Miller}
\maketitle
\begin{abstract}
  If all of the entries of a large $S_n$ character table are covered up 
  and you are allowed to uncover one entry at a time, then how can
  you quickly identify all of the indexing characters and conjugacy classes?
  We present a fast algorithmic
  solution that works even when $n$ is so large that 
  almost none of the entries of the character table can be computed. The fraction of the character table that needs to be uncovered has exponential decay, and for many of these entries we are only interested in whether the entry is zero. 
\end{abstract}

\thispagestyle{empty}
\section{Introduction}
\noindent
In \cite{ChowPaulhus}, T.\ Y.\ Chow and J.\ Paulhus
raised the following question. 
How can you~quickly identify an irreducible character $\chi$ of $S_n$
by querying its values $\chi(g)$ in order to determine
the partition $\lambda$ such that $\chi=\chi_\lambda$?
We call this the character identification problem.
They were interested in the number of queries required to identify $\chi$,
and they gave an algorithmic solution
where the number of queries is polynomial~in~$n$.

In \S3 of this note we give a different solution
to the character identification~problem. We show that any irreducible character of $S_n$ can be identified from at most $n$ values.
We introduce, for any character $\chi$ of $S_n$, a symbol.
If $\chi$ is irreducible, then
the symbol is the Frobenius symbol of a partition $\lambda$ and $\chi=\chi_\lambda$.
If the symbol is not a Frobenius symbol of a partition of $n$, then the character is not irreducible.
The symbol is inexpensive to compute
and relies on only a small number of values,
each of which can be computed quickly even for very large $n$ if $\chi$ is irreducible.

In \S5 we consider  
what might at first seem to
be a much more difficult problem.
Let $\chi_1,\chi_2,\ldots,\chi_{p_n}$ and $g_1^{S_n},g_2^{S_n},\ldots,g_{p_n}^{S_n}$ be the irreducible characters and conjugacy classes of $S_n$, written in some order,
and let $T$ be the character table of $S_n$ 
with respect to these orderings, so $T=[\chi_i(g_j)]_{1\leq i,j\leq p_n}$. 
If all of the entries of $T$ are covered up and 
you are allowed to uncover one entry at a time, then how do you quickly identify the characters $\chi_i$ and classes $g_j^{S_n}$?
We give an algorithmic solution to this problem, and 
in particular, we show that if $u(n)$ denotes the fraction of the character table that needs to be uncovered in order to identify all of the indexing characters and classes, then $u(n)$ has exponential decay as $n\to\infty$. 

\section{Notation}
\subsection{\!\!\!}
By a partition $\lambda$ of a positive integer $n$, we shall  
 mean an integer sequence 
 $(\lambda_1,\lambda_2,\ldots,\lambda_{\ell(\lambda)})$ such that $\lambda_1\geq \lambda_2\geq \ldots\geq \lambda_{\ell(\lambda)}>0$ and $\lambda_1+\lambda_2+\ldots+\lambda_{\ell(\lambda)}=n$.
 Associated to $\lambda$ is the usual partition $\lambda'$ with 
 parts $\lambda'_i=|\{k: \lambda_k\geq i\}|$, $1\leq i\leq \lambda_1$. 
 The number of partitions of $n$ is denoted~$p_n$ and asymptotic to
 $\frac{1}{4n\sqrt{3}}\exp(\frac{2\pi}{\sqrt{6}}\sqrt{n})$.

\subsection{\!\!\!}\label{abc notation}
Identifying a partition $\lambda$ with its diagram of boxes $b$, we also 
identify a box $b$ with its location $(i,j)$.
Let $\sq(\lambda)=\max(\{i: (i,i)\in\lambda\}\cup\{0\})$.
Associated to each box $b=(i,j)\in \lambda$ is 
the content $c(b)=j-i$, the hook
\[\H_{(i,j)}(\lambda)=\{(i,u): j\leq u\leq \lambda_i\}\cup\{(v,j): i\leq v\leq  \lambda'_j\},\]
the length of the hook $\h_{(i,j)}(\lambda)=|\H_{(i,j)}(\lambda)|$, and the partition
\[\lambda(i,j)=\{(u,v)\in \lambda : u\geq i \text{ and } v\geq j\}.\]
The principal hooks and their lengths are denoted
\[\H_i(\lambda)=\H_{(i,i)}(\lambda),\quad 
  {\h_i(\lambda)=\h_{(i,i)}(\lambda)},\quad i=1,2,\ldots,\sq(\lambda).\]
Let $\a_i(\lambda)=\lambda_i-i$ and $\b_i(\lambda)=\lambda'_i-i$, so 
$\a_i(\lambda)+\b_i(\lambda)+1=\h_i(\lambda)$, and let
\[\c_i(\lambda)=\sum_{b\in\lambda(i,i)} c(b).\]
Let $\a(\lambda)=(\a_1(\lambda),\ldots,\a_{\sq(\lambda)})$,
$\b(\lambda)=(\b_1(\lambda),\ldots,\b_{\sq(\lambda)})$,  
$\c(\lambda)=(\c_1(\lambda),\ldots,\c_{\sq(\lambda)})$,
and $\h(\lambda)=(\h_1(\lambda),\ldots,\h_{\sq(\lambda)})$.

\subsection{\!\!\!}
Given two integer sequences $a=(a_1,a_2,\ldots,a_k)$ and
$b=(b_1,b_2,\ldots,b_k)$, we
write $(a\mid b)=(a_1,a_2,\ldots,a_k\mid b_1,b_2,\ldots,b_k)$.
If
\begin{equation}\label{frob cond}
  a_1>a_2>\ldots>a_k\geq 0,\quad b_1>b_2>\ldots>b_k\geq 0,
\end{equation}
then $(a\mid b)$ is Frobenius's notation, called the Frobenius symbol, for 
the unique partition $\lambda$ of
 $\sum_{i=1}^k (a_i+b_i+1)$ 
such that $\lambda_i=a_i+i$ and $\lambda'_i=b_i+i$ for $1\leq i\leq k$.
If \eqref{frob cond} does not hold, then $(a\mid b)$ is
not the Frobenius symbol of a partition.

\subsection{\!\!\!}
If $\lambda$ and $\mu$ are partitions of  $n>0$, then $\chi_\lambda$ denotes the irreducible character of $S_n$ associated with $\lambda$ in the usual way, and $\chi_\lambda(\mu)$ denotes the value $\chi_\lambda(g)$ for any $g$ with cycle type $\mu$, meaning the parts of $\mu$ are the periods of the various cycles of~$g$.

\subsection{\!\!\!}
In \S\ref{char param} we use square brackets for partitions in order to emphasize that the following convention is in play. 
If $\nu=(\nu_1,\nu_2,\ldots,\nu_l)$ is a partition of $k\leq n$, then by 
$\chi([\nu_1,\nu_2,\ldots,\nu_l])$ we mean
$\chi((\nu_1,\nu_2,\ldots,\nu_l,1,1,\ldots,1))$ with $n-k$ appended $1$'s.

\section{Identifying characters}\label{char param}
In Definition~\ref{shape definition} we associate to any character $\chi$ 
a symbol $(\ax(\chi)\mid\bx(\chi))$, and in Theorem~\ref{id theorem}
we show that if $\chi\in\Irr(S_n)$, then the symbol is the Frobenius symbol of the partition $\lambda$ such that $\chi=\chi_\lambda$.
Important for us will be the result of Frobenius \cite[Eq.\ (16)]{Frobenius} that 
for any $\chi_\lambda\in\Irr(S_n)$, 
\begin{equation}\label{content sum}
  \c_1(\lambda)=\frac{\binom{n}{2}\chi_{\lambda}((1\, 2))}{\chi_\lambda(1)},
\end{equation}
where we take $\chi_{\lambda}((1\, 2))=0$ if $n=1$. 
See \cite[Thm.\ 3.1]{MillerMathAnn} for a more general result.

\begin{definition}\label{hd def}
  For any character $\chi$ of $S_n$, we define sequences
  \[{\hx(\chi)=(\hx_1(\chi),\hx_2(\chi),\ldots,\hx_{\kx(\chi)}(\chi))},\quad 
  {\dx(\chi)=(\dx_1(\chi),\dx_2(\chi),\ldots,\dx_{\kx(\chi)+1}(\chi))}\]
  recursively as follows.
  \begin{enumerate}[1)]
  \item $\dx_1(\chi)=\chi(1)$ and $\dx_2(\chi)=\chi([\hx_1(\chi)])$, where 
    $\hx_1(\chi)$ is the first element in the sequence
    $n_1:=n,n_1-1,\ldots,1$ such that $\chi([\hx_1(\chi)])\neq 0$.
  \item Given $\hx_1(\chi),\hx_2(\chi),\ldots,\hx_{u-1}(\chi)$ and
    $\dx_1(\chi),\dx_2(\chi),\ldots,\dx_u(\chi)$ with $u\geq 2$,
    \begin{enumerate}[a)]
    \item if $\sum_{i=1}^{u-1} \hx_i(\chi)=n$ or $\hx_{u-1}(\chi)\leq 2$,
      then $\kx(\chi)=u-1$;
    \item if $\sum_{i=1}^{u-1} \hx_i(\chi)<n$ and $\hx_{u-1}(\chi)>2$,
      then
      ${n_u=
      \min\{\hx_{u-1}(\chi)-2,n_{u-1}-\hx_{u-1}(\chi)\}}$,
      $\hx_u(\chi)$ is the first element of 
      $n_u,n_u-1,\ldots, 1$ such that 
      ${\chi([\hx_1(\chi),\hx_2(\chi),\ldots, \hx_u(\chi)])\neq 0}$, 
      and
      $\dx_{u+1}(\chi)=\chi([\hx_1(\chi),\hx_2(\chi),\ldots, \hx_u(\chi)])$.
    \end{enumerate}
  \end{enumerate}
\end{definition}

\begin{lemma}\label{hook lemma}
  If $\chi=\chi_\lambda\in\Irr(S_n)$, then
    \[\kx(\chi)=\sq(\lambda),\quad \hx_i(\chi)=\h_i(\lambda),\quad
    \dx_i(\chi)=\pm\chi_{\lambda(i,i)}(1),\quad
    i=1,2,\ldots,\kx(\chi).\]
\end{lemma}

\begin{proof}
  This is by the Murnaghan--Nakayama rule, since for 
  any nonempty partition $\nu$, the hook $\H_{(1,1)}(\nu)$
  has more boxes than any other hook $\H_b(\nu)$ with $b\neq(1,1)$, and
  it has at least two more boxes than any other 
  principal hook $\H_{(j,j)}$, $2\leq j\leq\sq(\nu)$. 
\end{proof}

\begin{definition}\label{cx def}
  For any character $\chi$ of $S_n$, let
  $\cx(\chi)=(\cx_1(\chi),\cx_2(\chi),\ldots,\cx_{\kx(\chi)}(\chi))$,
  \[
    \cx_i(\chi)=
    \frac{\binom{n-\sum_{j=1}^{i-1}\hx_j(\chi)}{2}\chi([\hx_1(\chi),\hx_2(\chi),\ldots, \hx_{i-1}(\chi) ,2])}{\dx_i(\chi)},\quad i=1,2,\ldots,\kx(\chi),
  \]
  with the convention that $\cx_{\kx(\chi)}(\chi)=0$ if $\hx_{\kx(\chi)}(\chi)=1$.
\end{definition}

\begin{lemma}\label{content lemma}
If $\chi=\chi_\lambda\in\Irr(S_n)$, then $\cx(\chi)=\c(\lambda)$.
\end{lemma}

\begin{proof}
  By Lemma~\ref{hook lemma} and Eq.\ \eqref{content sum}.
\end{proof}

\begin{definition}\label{shape definition}
  For any character $\chi$ of $S_n$, we define
  $\ax(\chi)=(\ax_1(\chi),\ax_2(\chi),\ldots,\ax_{\kx(\chi)}(\chi))$
  recursively as follows.
  \begin{enumerate}[{\rm 1)}]
  \item $\ax_{\kx(\chi)}(\chi)=\frac{1}{\hx_{\kx(\chi)}(\chi)}(\cx_{\kx(\chi)}(\chi)+\binom{\hx_{\kx(\chi)}(\chi)}{2})$. \label{last arm part}
  \item For $i=\kx(\chi)-1,\kx(\chi)-2,\ldots, 1$, 
    \[
      \begin{split}
\ax_i(\chi)
        =
        \ax_{i+1}(\chi)&+\frac{1}{\hx_i(\chi)}\left[\cx_i(\chi)-\cx_{i+1}(\chi) -\binom{\ax_{i+1}(\chi)+1}{2}+\binom{\hx_{i+1}(\chi)-\ax_{i+1}(\chi)}{2}\right.\\
        &\ \ \left.+\binom{\hx_i(\chi)-\hx_{i+1}(\chi)}{2}+ (\hx_{i+1}(\chi)-\ax_{i+1}(\chi)-1)(\hx_i(\chi)-\hx_{i+1}(\chi))\right].
      \end{split}\label{ai rec eq}
    \]\label{rec part}
  \end{enumerate}
  We define $\bx(\chi)=(\bx_1(\chi),\bx_2(\chi),\ldots,\bx_{\kx(\chi)}(\chi))$, where 
  $\bx_i(\chi)=\hx_i(\chi)-\ax_i(\chi)-1$.
\end{definition}

\begin{theorem}\label{id theorem}
  If $\chi\in\Irr(S_n)$, then $(\ax(\chi)\mid \bx(\chi))$
  is the Frobenius symbol of a partition $\lambda$ and $\chi=\chi_\lambda$.
\end{theorem}

\begin{proof}
  Suppose $\chi=\chi_\lambda\in\Irr(S_n)$. 
  The assertion is then  
  ${(\ax(\chi)\mid \bx(\chi))=(\a(\lambda)\mid \b(\lambda))}$.
  By Lemma~\ref{hook lemma}, $\hx(\chi)=\hx(\lambda)$,
  so the claim reduces to $\ax(\chi)=\a(\lambda)$.

  Let (I) and (II) denote the statements
  obtained by replacing $\chi$ by $\lambda$ everywhere in 
  parts 1) and 2) of Definition~\ref{shape definition}.
  Since $\hx(\chi)=\h(\lambda)$ and $\cx(\chi)=\c(\lambda)$ by
  Lemmas~\ref{hook lemma} and~\ref{content lemma}, we have that if (I) and (II) both hold, then  $\ax(\chi)=\a(\lambda)$.

  For (I), we have
  \[\c_{\sq(\lambda)}(\lambda)
    =\binom{\a_{\sq(\lambda)}(\lambda)+1}{2}-\binom{\h_{\sq(\lambda)}(\lambda)-\a_{\sq(\lambda)}(\lambda)}{2}
    =\a_{\sq(\lambda)}(\lambda) \h_{\sq(\lambda)}(\lambda) -\binom{\h_{\sq(\lambda)}(\lambda)}{2}.\]

  For (II), we have
    \[
    \c_i(\lambda)-\c_{i+1}(\lambda)=\sum_{b\in \H_i(\lambda)}c(b)
                 =
                   \sum_{b\in \H_{i+1}(\lambda)}c(b)
                   +\sum_{{b\in \H_i(\lambda)}\atop{c(b)>\a_{i+1}(\lambda)}}c(b)
                   +\sum_{{b\in \H_i(\lambda)}\atop{c(b)<-\b_{i+1}(\lambda)}}c(b).
  \]
  So, with $s=\a_i(\lambda)-\a_{i+1}(\lambda)$ and $t=\b_i(\lambda)-\b_{i+1}(\lambda)$, 
  \begin{equation}
    \begin{split}
      \c_i(\lambda)
      = \c_{i+1}(\lambda)
      +\binom{\a_{i+1}(\lambda)+1}{2}&-\binom{\b_{i+1}(\lambda)+1}{2}\\
      &\!\!\!\!\!\!\!\!+
      \binom{s+1}{2} +  s \a_{i+1}(\lambda)
      - \binom{t+1}{2} -t \b_{i+1}(\lambda).
    \end{split}\label{content equation}
  \end{equation}
  Now $t=\h_i(\lambda)-\h_{i+1}(\lambda)-s$, so
  \[\binom{s+1}{2}-\binom{t+1}{2}=(1+\h_i(\lambda)-\h_{i+1}(\lambda))s-\binom{\h_i(\lambda)-\h_{i+1}(\lambda)+1}{2}.\]
  Further substituting elsewhere in \eqref{content equation} the relations
  $s=\a_i(\lambda)-\a_{i+1}(\lambda)$, $t=(\h_i(\lambda)-\h_{i+1}(\lambda))-(\a_i(\lambda)-\a_{i+1}(\lambda))$, and $\b_{i+1}(\lambda)=\h_{i+1}(\lambda)-\a_{i+1}(\lambda)-1$,
  so that both sides of \eqref{content equation} are  entirely in terms of $\a_l(\lambda)$'s, $\h_l(\lambda)$'s, and $\c_l(\lambda)$'s, 
  and then solving for $\a_{i+1}(\lambda)$,
  we obtain (II).
\end{proof}

\begin{corollary}\label{non irr}
If $\chi$ is a character of $S_n$ and $(\ax(\chi)\mid\bx(\chi))$ is not the Frobenius symbol of a partition of $n$, then $\chi$ is not irreducible.
\end{corollary}

\begin{proof}
By Theorem~\ref{id theorem}
  \end{proof}

\begin{proposition}\label{query prop}
  $(\ax(\chi)\mid\bx(\chi))$
  requires $\leq n$ values $\chi(g)$ for any character $\chi$ of~$S_n$.
\end{proposition}
\begin{proof}
  If $\hx_{\kx(\chi)}(\chi)\geq 2$, then the number of
  values $\chi(g)$ that go into $\hx(\chi)$ and $\dx(\chi)$~is
  \[1+(n-\hx_1(\chi)+1)+\sum_{i=1}^{\kx(\chi)-1}(\hx_i(\chi)-\hx_{i+1}(\chi)-1)
    =n-\hx_{\kx(\chi)}(\chi)-\kx(\chi)+3.\]
  If $\hx_{\kx(\chi)}(\chi)<2$, then $\hx(\chi)$ and $\dx(\chi)$ together use
  $n-\hx_{\kx(\chi)}(\chi)-\kx(\chi)+2$ queries, since
  after querying $\chi([\hx_1(\chi),\hx_2(\chi),\ldots,\hx_{\kx(\chi)-1}(\chi),2])=0$, it follows
  that $\hx_{\kx(\chi)}(\chi)=1$ and $\dx_{\kx(\chi)}(\chi)=\dx_{\kx(\chi)-1}(\chi)$, i.e.\
   $\dx_{\kx(\chi)}(\chi)$ was already queried earlier in the algorithm.

   If $\hx_{\kx(\chi)}(\chi)>2$, then $\cx(\chi)$
   uses $\kx(\chi)$ additional queries.
   If $\hx_{\kx(\chi)}(\chi)\leq 2$,
   then $\cx(\chi)$
   uses $\kx(\chi)-1$ queries,
   since $\cx_{\kx(\chi)}(\chi)=0$ if $\hx_{\kx(\chi)}(\chi)=1$,
   and the values that go into 
   $\cx_{\kx(\chi)}(\chi)$ were already queried for $\dx(\chi)$ if $\hx_{\kx(\chi)}(\chi)=2$.
   
   So if $q(\chi)$ denotes the number of queries to construct $(\ax(\chi)\mid \bx(\chi))$, then
   \begin{equation}
     q(\chi)=\begin{cases}
       n-\hx_{\kx(\chi)}(\chi)+3
       & \text{if $\hx_{\kx(\chi)}(\chi)\geq 3$,}\\
       n & \text{if $\hx_{\kx(\chi)}(\chi)\leq 2$.}
     \end{cases}
   \end{equation}
 \end{proof}

\begin{remark}
  For $\chi\in\Irr(S_n)$, all of the queried values $\chi(g)$
    that go into computing $(\ax(\chi)\mid\bx(\chi))$
    are trivial to compute even if $n$ is very large, since each
    is either trivially zero in the sense of \cite{MillerScheinerman}
    (for the $\hx_i(\chi)$'s), can be
    quickly computed by using the so-called hook-length formula \cite{JamesKerber,Macdonald} (the $\dx_i(\chi)$'s), or can be quickly computed using the formula given by Definition~\ref{cx def} and Lemma~\ref{content lemma} (the $\cx_i(\chi)$'s).
\end{remark}

\section{Identifying classes}
\noindent
We now deal with the natural dual problem of determining a class $g^{S_n}$ from
a set of values~${\chi_\lambda(g)}$.
The solution is
given by the following well-known lemmas.
Let
\[
  \xi_{n,k}=\chi_{(n-k,1,1,\ldots,1)}\in\Irr(S_n),\quad k=0,1,2,\ldots,n-1.
\]
If $V$ is an $S_n$-module, let $\chi_V$ denote the character
afforded by $V$. 
Given a partition $\lambda$ of $n$,
let $V_\lambda$ denote the $S_n$-module affording 
$\chi_\lambda\in \Irr(S_n)$.

\begin{lemma}\label{ext hooks}
$\xi_{n,k}=\chi_{\wedge^k V_{(n-1,1)}}$ for $k=0,1,2,\ldots,n-1$.
\end{lemma}

\begin{proof}
  This follows by inducting on $n$ and using the branching rule. 
  The case $n=1$ is clear. Assuming  
  the claim for a fixed $n\geq 1$ and considering the next case,
  certainly $\wedge^kV_{(n,1)}=V_{(n+1)}$ if $k=0$.
  If $1\leq k\leq n$, then as $S_n$-modules,
  \[\wedge^k V_{(n,1)}\cong \wedge^k(V_{(n)}\oplus V_{(n-1,1)})
    \cong \wedge^{k-1} V_{(n-1,1)}\oplus \wedge^{k} V_{(n-1,1)},\]
  and by hypothesis the right-hand side is
  isomorphic to 
  $V_{(n-k+1,1,1,\ldots,1)}\oplus V_{(n-k,1,1,\ldots,1)}$ if $k\neq n$, and $V_{(n-k+1,1,1,\ldots,1)}$ if $k=n$.
  By the branching rule again, this implies that, as $S_{n+1}$-modules, 
  $\wedge^k V_{(n,1)}\cong V_{(n+1-k,1,1,\ldots,1)}$ for $k=0,1,\ldots,{n+1}$.
\end{proof}

\begin{lemma}\label{class id lemma}
  For any partition $\nu$ of $n$,
  \[(X-1)\sum_{k=0}^{n-1} (-1)^k\xi_{n,k}(\nu) X^{n-1-k}=\prod_{k=1}^{\ell(\nu)}(X^{\nu_k}-1).\]
  In particular, $\nu$ is determined
  by the values $\xi_{n,n-1}(\nu),\xi_{n,n-2}(\nu),\ldots,\xi_{n,n-\lfloor\frac{n}{2}\rfloor}(\nu)$.
\end{lemma}

\begin{proof}
  Let $p(X)=\det(X-gI)$ for an element $g$ with cycle type $\nu$ under
  the permutation representation $V=\mathbb C^n$ given by
  $\sigma\mapsto (\delta_{i\sigma(j)})_{i,j}$, so 
  ${p(X)=\prod_{k=1}^{\ell(\nu)} (X^{\nu_k}-1)}$.  
  Decomposing $V$ as $V_{(n)}\oplus V_{(n-1,1)}$ and denoting the
  eigenvalues of $g$ on $V_{(n-1,1)}$ by $\zeta_1,\zeta_2,\ldots,\zeta_{n-1}$, then
  \[\frac{p(X)}{X-1}
    =
    \prod_{k=1}^{n-1}(X-\zeta_k)
    =
    \sum_{k=0}^{n-1} (-1)^k \chi_{\wedge^kV_{(n-1,1)}}(g) X^{n-1-k}=\sum_{k=0}^{n-1}(-1)^k\xi_{n,k}(\nu)X^{n-1-k},\]
  where the last equality is by Lemma~\ref{ext hooks}.
\end{proof}

\section{Identifying both characters and classes}
Throughout this section, by character table of $S_n$ we shall mean the table
\[[\chi_i(g_j)]_{1\leq i,j\leq p_n}\]
for some arbitrary but fixed ordering of the irreducible characters
$\chi_1,\chi_2,\ldots,\chi_{p_n}$ and classes $g_1^{S_n},g_2^{S_n},\ldots,g_{p_n}^{S_n}$.
The object is to identify the characters and classes by successively querying entries $T_{ij}$ (i.e., querying the function $(i,j)\mapsto T_{ij}$), which is possible only if $n\neq 4,6$ by inspection. In more visual terms, all of the entries of $T$ are covered up and the object is to identify the characters $\chi_i$ and classes $g_j^{S_n}$ by successively uncovering various entries $T_{ij}$. 

We start with two well-known facts in Lemmas~\ref{unique max} and~\ref{n-1}.

\begin{lemma}\label{unique max}
  If $n\neq4$ and $\chi\in\Irr(S_n)$ with $\chi(1)>1$, then
  $\chi(1)>|\chi(g)|$ for $g\neq 1$.
\end{lemma}
\begin{proof}
See for example \cite[Lemma~6]{MillerPisa}.
\end{proof}

\begin{lemma}\label{n-1}
  If $n\neq 6$ and $\chi_\lambda\in\Irr(S_n)$ with $\chi_\lambda(1)>1$, then
  $\chi_\lambda(1)\geq n-1$ with equality if and only if $\lambda\in\{(n-1,1),(2,1,1,\ldots,1)\}$.
\end{lemma}

\begin{proof}
  By induction on $n$ and the branching rule.
\end{proof}

\begin{lemma}\label{detect hook}
  If $\lambda$ is a partition of $n\geq 4$, then $\lambda=(n-k,1,1,\ldots,1)$ for some $k$
  if and only if
  $\chi_{\lambda}(1)=4\chi_\lambda((1\, 2\, 3))-3\chi_\lambda((1\,2)(3\,4))$.
\end{lemma}

\begin{proof}
  By the branching rule, $\lambda=(n-k,1,1,\ldots,1)$ for some $k$ if and only if $\langle {\rm Res}_{S_4}\chi_\lambda,\chi_{(2,2)}\rangle=0$,
i.e.\ if and only if $\chi_{\lambda}(1)=4\chi_\lambda((1\, 2\, 3))-3\chi_\lambda((1\,2)(3\,4))$.
\end{proof}

\begin{proposition}\label{discover hooks}
  If $n\neq 4,6$ and all of the entries of the $S_n$ character table ${T=[\chi_i(g_j)]_{1\leq i,j\leq p_n}}$ are covered up, then 
  the row positions of the
  characters
  \[\xi_k=\chi_{(n-k,1,1,\ldots,1)}\in\Irr(S_n),\quad 0\leq k\leq n-1,\]
  can 
  be determined by uncovering at most $7p_n+n$ 
  entries.
\end{proposition}

\begin{proof}
  The object is to determine,
  among the $p_n$ rows of the character table $T$,
  the positions $i_j$ of the characters
  $\xi_j=\chi_{(n-j,1,1,\ldots,1)}\in\Irr(S_n)$ with $0\leq j\leq n-1$.
  Along the way it will be important to 
  determine, among the $p_n$ columns, the positions $a,b,c,d$ of
  the classes $1^{S_n},(1\, 2)^{S_n},(1\, 2\, 3)^{S_n},(1\,2)(3\,4)^{S_n}$.
  We may assume $n>6$ by inspection, since $p_n^2<7p_n+n$ for $n\leq 5$.
  
  \begin{enumerate}[Step 1.]
  \item Determine the column $a$ corresponding to the identity element.
    This can be done with $\leq 3p_n$ moves by first uncovering an
    entire row of entries for some non-linear irreducible character.
    This requires at most $3p_n$ moves because there are exactly
    two linear characters.
    By Lemma~\ref{unique max}, the row for any nonlinear character
    contains a unique
    largest entry in position $a$.

  \item Uncover the entries in column $a$ to reveal all of the
    character degrees $\chi(1)$,
    and then locate the pair of row indices $S=\{r<s\}=\{i_1,i_{n-2}\}$
    that correspond to
    the unique characters $\xi_1,\xi_{n-2}$ of degree $n-1$.
    There are exactly two such characters by Lemma~\ref{n-1}. 
    The cost of uncovering column $a$ is $\leq p_n-1$ queries.

  \item 
    Uncover new entries in row $r$ until there is one entry
    with absolute value $n-3$, one entry with absolute value $n-4$,
    and two entries with absolute value $n-5$.
    This is possible because, for any partition $\mu$ of $n$,
    \begin{equation}\label{xi_1}
      \xi_1(\mu)=(-1)^{n-\ell(\mu)}\xi_{n-2}(\mu)=|\{i: \mu_i=1\}|-1.
    \end{equation}
    Denote by $j_1,j_2,j_3,j_4$
    the positions of these entries in row $r$, with $j_1$ corresponding
    to the entry with absolute value $n-3$, and $j_2$ corresponding to
    the entry with absolute value $n-4$. 
    From these entries and \eqref{xi_1} we immediately determine
    $i_1,i_{n-1},b,c,d$. We have $b=j_1$, $c=j_2$,  
    \[
      i_1=\begin{cases}
        r & \text{if $\chi_r(g_{j_1})$ is positive,}\\
        s & \text{otherwise,}
      \end{cases}
    \]
    $i_{n-1}$ is the single element of $S\smallsetminus \{i_1\}$, and 
    \[
      d=\begin{cases}
        j_3 & \text{if $\chi_{i_{n-1}}(g_{j_3})$ is positive,}\\
        j_4 & \text{otherwise.}
      \end{cases}
    \]
    This step requires uncovering $\leq p_n-1$ new entries.
    
    \item Use Lemma~\ref{detect hook} to construct the set of desired positions
      $I=\{i_0,i_1,\ldots,i_{n-1}\}$ by
      starting with $I=\emptyset$ and 
      successively going through the rows $i=1,2,\dots$ and including
      an index $i$ in $I$ if $\chi_i(g_a)=4\chi_i(g_c)-3\chi_i(g_d)$,
      until $|I|=n$.
      This requires uncovering at most $2p_n-2$ new entries.

    \item Let $f(i)=\frac{\binom{n}{2}\chi_i(g_b)}{\chi_i(g_a)}$, so if $\chi_i=\chi_\lambda$, then $f(i)=\c_1(\lambda)$ by \eqref{content sum}. 
      We then obtain the desired sequence $i_0,i_1,\ldots,i_{n-1}$ from
      the set $I=\{i_0,i_1,\ldots,i_{n-1}\}$
      by using the fact that $f(i_0)>f(i_1)>\ldots>f(i_{n-1})$.
      Computing these $f(i_j)$'s requires at most $n-2$ additional
      entries $\chi_{i_j}(g_b)$ be uncovered.\qedhere
    \end{enumerate}
  \end{proof}

\begin{theorem}\label{uncovering numbers}
  If $n\neq 4,6$ and the entries of the $S_n$
  character table $[\chi_i(g_j)]_{1\leq i,j\leq p_n}$ are covered up, 
  then
  \begin{enumerate}[1)]
  \item the classes $g_j^{S_n}$ can be determined
    by uncovering at most $\lfloor \frac{n}{2}\rfloor p_n+7p_n+n$ entries,
  \item the characters $\chi_i$ can be identified
    by uncovering at most $np_n$ additional entries.
  \end{enumerate}
\end{theorem}
  \begin{proof}
    The first part follows from Lemma~\ref{class id lemma} and Proposition~\ref{discover hooks}.
    The second part follows from the first part, since once the classes $g_j^{S_n}$ are known, each character $\chi_i$ can be identified by uncovering at most $n$ additional entries by Theorem~\ref{id theorem} and the query calculation in Proposition~\ref{query prop}. 
  \end{proof}

  \begin{theorem}
    If $n\neq 4,6$ and the entries of the $S_n$ character table
    $[\chi_i(g_j)]_{1\leq i,j\leq p_n}$
    are covered up,
    and if $u(n)$ denotes the
    fraction of
    the character table that 
    needs to be uncovered
    in order to
    identify all of the indexing characters $\chi_i$ and classes $g_j^{S_n}$,  then
    $u(n)=O(n/p_n)$ has exponential decay as $n\to\infty$.
  \end{theorem}

  \begin{proof}
    By Theorem~\ref{uncovering numbers}.
  \end{proof}

  \begin{remark}
  There are two alternative approaches that one might try for 
  discovering the positions of the hook-shaped characters in Proposition~\ref{discover hooks},
  but both of these approaches come with difficulties.
  It is natural to try using either the column indexed by $1^{S_n}$ or the column indexed by $(1\, 2\, \ldots\, n)^{S_n}$.
  The trouble with the identity column is that in general there can be many characters with the same degree. In fact, for any positive integer $k$, and for sufficiently large $n$, there exist $k$ irreducible characters of $S_n$ with the same degree. (Interestingly, there seems to be an unexplored dual version of this result where character degrees are replaced by class sizes.)
  However, it seems to be the case that the set of hook-shaped characters can indeed be identified using degrees. At least experimentally, with the exception of $n=6,12,15,24,35$, it seems to be the case that the 
  characters $\chi_{(n-k,1,1,\ldots,1)}\in\Irr(S_n)$, $0\leq k\leq n-1$, are
  the only irreducible characters $\chi$ of $S_n$ such that 
  $\chi(1)\in\{\binom{n-1}{k}: 0\leq k\leq n-1\}$. It would be interesting to know if there are only a finite number of exceptional cases. Regarding the column
  for $(1\, 2\, \ldots\, n)^{S_n}$, the problem here is that there are multiple so-called ``sign partitions,'' i.e.\ partitions $\mu$ such that $\chi_\lambda(\mu)\in\{0,1,-1\}$ for all $\chi_\lambda\in\Irr(S_n)$ \cite{Morotti}. Nevertheless, with the exception of only a few special cases for small $n$, the column $(1\, 2\, \ldots\, n)^{S_n}$ can be identified by additionally considering the number of nonzero entries. I will omit the proof of the following lemma. 
    \begin{lemma}
    Let $\mu$ be a partition of $n$. Then the following are equivalent.
    \begin{enumerate}[(i)]
    \item $\chi_\lambda(\mu)\in\{0,1,-1\}$ for all
      $\lambda$ and $\sum_{\nu}|\chi_\nu(\mu)|^2=n$.
    \item $\mu=(n)$, $\mu=(3,2,1)$, $\mu=(2,1,1)$, or $\mu=(1,1)$.
    \end{enumerate}
  \end{lemma}
  \noindent
  So the problem of quickly locating the column indexed by
  $(1\, 2\, \ldots\, n)^{S_n}$ reduces to quickly locating a column with
  the right number of $\pm 1$'s and then $0$'s elsewhere.
  But the fraction of
  the character table covered by $0$'s and $\pm 1$'s is not
  yet understood. The result of \cite{PeluseSoundararajan} shows that
  the fraction of the table covered by $\pm 1$'s tends to $0$ as $n\to\infty$.
  The limiting behavior for the fraction of the table covered by $0$'s is not known, see \cite{MillerMathZ,MillerJCTA,MillerScheinerman}.
\end{remark}

\end{document}